\documentclass[12pt]{amsart}
\usepackage{amsmath, amsfonts, amscd, amssymb, amsthm, amstext, amsopn, amsbsy}
\usepackage{hyperref}

\newtheorem{Theorem}{Theorem}[section] \newtheorem{Lemma}[Theorem]{Lemma}
\newtheorem{Corollary}[Theorem]{Corollary} 

\newtheorem{Remark}[Theorem]{Remark}

\def\bc{\bold C}

\def\vt{\vert}

\def\fr{\frac}

\begin{document}
\title{A note on  function algebras on disks }
\author  {Kieu Phuong Chi and Mai The Tan}
\address{Kieu Phuong Chi,  \newline
Department of Mathematics and Applications, Saigon  University, 273  An Duong Vuong, Ho Chi Minh City, Vietnam}
\email{\href{mailto:kieuphuongchi@sgu.edu.vn}{kieuphuongchi@sgu.edu.vn}}
   \address{Mai The Tan,  \newline
   Department of Mathematics and Applications, Saigon  University, 273  An Duong Vuong, Ho Chi Minh City, Vietnam}
   \email{\href{mailto:mttan.q11@hcm.edu.vn}{mttan.q11@hcm.edu.vn}}


\begin{abstract} Let $D$ be a closed disk in the complex plane centered at the origin,  $f, g$  complex valued continuous function on $D$. Let $P[f,g; D]$ (res. $R[f, g; D])$) be the uniform closure on $D$ of polynomials (res. rational functions) in variables $f$ and $g$. In \cite{OS}, using complex dynamical systems,  O'Farrell and Sanabria-Garcia proved  that $\{\Big(z^2, \cfrac{\overline z}{1+\overline{z}}\Big): z\in D\}$  is not polynomially convex with $D$ small enough and so that  $P[z^2,\cfrac{\overline z}{1+\overline z}; D]\ne C(D)$ if $D$ is sufficient small. In this paper, we first give a certain conditions for rational convexity of union of two compact set of $\Bbb C^n$ and apply to show that  $R[z^2, \cfrac{\overline z}{1+\overline z}; D]= C(D)$ for all $D$ small enough.
 \end{abstract}

\maketitle

\section{Introduction and statement of results}

In this note, we refine techniques from previous work of both authors \cite{DC,P1,P2,PW}..., to
obtain a new approximation result in the theory of function algebras.
In the papers \cite{D1,DP,DC,PW} the following situation is investigated. Let the function $f$ be a $\mathcal C^1$ function defined in a neighborhood of the origin in the complex plane which satisfies
$f(0)= 0,$ $\cfrac{\partial f}{\partial z}(0) = 0$ , $\cfrac{\partial f}{\partial \overline z}(0)=1$, (i.e., $f$ looks like $\overline z$ near $0$), and such that $z^m$ and $f^n$ separate points near $0$, where $m, n$ is positive integers. Let $D$ be a small closed disk in the complex plane, centered at the origin. Is it possible to uniformly approximate every continuous function on $D$ by polynomials in $z^m$ and $f^n$? It is shown that both answers yes and no are possible. Using some technical lemmas and the well-known result of H\"ormander and Wermer on local approximation in totally real manifolds, De Paepe proved that $P[z^m, f^n; D]=C(D)$ if $m, n$ are coprime. If $m$ and $n$ are not coprime (and the generators separate points near the origin), then one can reduce the problem to the case where $m = n$. Using polynomial convexity theory, it can be shown that $[z^2, f^2;D] = C(D)$ for some choices of $f$ (see \cite{P2,P3,PW}...). In \cite{P3}, the nice example are given to show that for $m = n = 2$, the answer to our original question can be negative. More precisely,   $P[z^2,f^2;D]\ne C(D)$ for $f(z)=\cfrac{\overline z}{1+\overline z}$ for sufficiently small disk $D$. The crucial point in showing that $P[z^m, f^n; D]= C(D)$ is to show that a certain set which is the union of two polynomially convex disks is itself polynomially convex. For this, one have to apply appropriate  tool of Kallin's lemma and some its modified forms.  The aim of this work is to expose the algebra $R[z^m, f^n; D]$. We would like to emphasize that $R[z^m, f^n; D]=C(D)$ when $P[z^m, f^n; D]=C(D)$. It now follows that the problem is meaningful in viewing $P[z^m, f^n; D]\ne C(D)$.  In \cite{OS}, the authors show that  all disc of the form $X = \{\big(z^2; f(\overline{z})\big) : z\in D\}$ is not polynomially convex, where $f $is holomorphic on $D$, and $f(z) = z^2 +a_3 z^3 +...$ with all coefficients $a_n$ real, and at least one $a_{2n+1}\ne 0$. By this fact, we can conclude that
$$P[z^2, \cfrac{\overline z}{1+\overline z}; D]\ne C(D).$$ In this work, we first prove the version  of Stout's theorem for rational function algebra and apply to  show that  the following
\begin{Theorem}\label{main4} If $D$ is a sufficiently small disk centered at the origin then
$$R[z^2, \cfrac{\overline z}{1+\overline z}; D]=C(D).$$
\end{Theorem}
To prove this fact, we continue to rely heavily on the theory of polynomial convexity and rational convexity. It may be useful to recall the general scheme in proving $[z^m, g; D]=C(D)$ for appropriately chosen $g$. Roughly speaking we consider the compact set $\tilde X$ which is inverse of $X := \{(z^m, g): z \in D\}$
under the proper polynomial mapping $(z, w) \mapsto (z^m, w)$. Then  $\tilde X$ is a union of graphs (in $\Bbb C^2$) over $D$. The key of  the proof  that to prove  $\tilde X$ is rationally convex. For this, we shall  use an appropriate tool which is the  version of Kallin's lemma (in fact Stout's theorem for rational algebras). Note  that union of two polynomially convex sets  may even fail to be  rationally convex (see \cite{Stol} ).  Moreover, Stolzenberg shown that  there are two totally real set in $\Bbb C^2$ that their union is not rational convex.
\par
We would like to note that Preskenis (see \cite{Pre}) give  some sufficient conditions such that $R[z, f(z); D]=C(D)$ (but not $z^2$), where $f$ is polynomial of degree 2, in two variables in $z, \overline{z}  $.

We do not to know what under conditions $$R[z^2, f(z)^2; D]=C(D)$$ that follows  $$P[z^2, f(z)^2; D]=C(D),$$ where $f$ looks like $\overline z$ near $0$.
\vskip 0.5cm
\noindent
{\bf Acknowledgment.} This work was done  during a stay of the first author at Vietnam Institute for Advance Study in Mathematics. He wishes to express his gratitude to the institute for the support.

\section{Preliminaries}

\hskip 0.6cm  For a compact set $K\subset \bc^n$, let $C(K)$ denote the algebra
 of all continuous complex valued function on $K$, with norm
$$\Vert g\Vert_K=\max\{\vt g(z)\vt : z\in K\}, \quad \text{ for
 every}\quad g\in C(K),$$
and let $P(K)$ denote the closure of set polynomials in $C(K)$; let
 $A(K)$ be the subalgebra of  $C(K)$ of the  functions which are
 holomorphic on the interior $\text{int}(K)$ of $K$; let $ R(K)$ be the closure in
 $C(K)$ of the rational functions with poles off $K$.
Let  $K$ be a compact subset of $\Bbb C^n$, by $\hat K$ we
 denote the polynomially convex hull of $K$ i.e.,
$$\hat K=\{z \in \bold C^n: \vert p(z)\vert \le \max_K \vert p\vert  \
  \text{for every polynomial}\  p \  \text{in}\  \Bbb C^n \}.$$
We say that $K$ is polynomially convex if $\hat K=K$. By definition,
 $R\text{-hull}(K)$ consists of all $z\in\Bbb C^n$ such that
$$\vert g(z)\vert \le\max_{K}\vert g\vert$$
for every  rational function $g$ which is analytic about $K$. If
 $K=R\text{-hull} (K)$, we say that $K$ is rationally convex in $\Bbb C^n$.
 Notice that $K\subset R\text{-hull}(K)\subset \hat {K}.$ Moreover, these inclusions may be proper. It is well-known that $K,\hat K$ and  $R\text{-hull}(K)$ respectively can be identified with the space of maximal ideal of $C(K), P(K)$ and $R(K)$.
\par The interest for studying polynomial convexity and rational convexity stems from the celebrated
Oka-Weil approximation theorem (see \cite{AW}, page 36) which states that holomorphic functions near
a compact polynomially (resp. rationally) convex subset  of $\Bbb C^n$ can be uniformly approximated by polynomials (resp. rational functions) in
$\Bbb C^n$. Later on, H\"ormander and Wermer proved that continuous function on  a compact polynomially convex subset of smooth {\it totally real} manifold $M$ in $\Bbb C^n$ can be approximated uniformly by polynomials (see Theorem 1.1, \cite{HW}). Recall that a manifold $M$ is totally real at $p\in M$  if the real tangent space $T_p(M)$ of $M$ at $p$ contains no complex line. A manifolds $M$ is totally real if it is totally real at any point of $M$.  An example of totally real manifold is the real Euclidean space $\Bbb R^n$. In this case, the mentioned  above theorem of H\"ormander and Wermer  reduces to classical  Stone-Weierstrass theorem. Note that, if $f_1,..., f_n$ are  $\mathcal C^1$ functions on an open subset $U$ of $\Bbb C^n$ and $\det\Big(\cfrac{\partial f_i}{\partial \overline z_j}(a)\Big)_{ij}\ne 0$ with $a=(a_1,..., a_n)\in U$  then $ M=\{(z_1,..., z_n, f_1,..., f_n): z\in U \}$ is totally real at $(a_1,.., a_n, f_1(a),..., f_n(a))$. The reader may consult  excellent sources like \cite{AW} and \cite{St} for more applications of polynomial convexity and rational convexity to function theory of several complex variables.
\par Observe that union of two polynomially convex sets may even fail to be  rationally convex (see\cite{Stol} ). On the positive side, the following result due to Kallin gives a sufficient condition for polynomial convexity of union of two polynomially convex compact sets. It is a powerful tool in verifying polynomial convexity of finite union of polynomially convex sets.
 \par
 \begin{Theorem} {\rm ( Kallin's lemma \cite{Ka,P,St})}  Suppose that:
\par \noindent 1) $X_1$ and $X_2$ are polynomially convex subsets of
 $\bc^n$;
\par \noindent 2) $Y_1$ and $Y_2$ are polynomially convex subsets of
 $\bc$ such that $0$ is a boundary point of both $Y_1$  and $Y_2$ and
 $Y_1\cap Y_2= \{0\}$;
\par \noindent 3) $p$ is a polynomial such that $p(X_1)=Y_1$ and
 $p(X_2)=Y_2$;
\par \noindent 4) $p^{-1}(0)\cap (X_1\cup X_2)$ is polynomially convex.
\par \noindent Then $X_1\cup X_2$ is polynomially convex.
\end{Theorem}
In \cite{Chi}, we provided an analogous result for rational convexity. Stout improved the conclusion of Kallin's lemma by strengthening the conditions.
  \begin{Theorem}\label{stout} {\rm (Stout's theorem \cite{P,St})}  Suppose that:
\par 1) $X_1, X_2$ are compact subset of $\Bbb C^n$ with $P(X_1)=C(X_1)$ and $P(X_2)=C(X_2)$.
\par 2) $Y_2, Y_2$ are polynomially convex sets of $\Bbb C$ such that $0$ is a boundary point of both  $Y_1$ and $Y_2$, and $Y_1\cap Y_2=\{0\}$.
\par 3) $p$ is a polynomial on $X_1\cup X_2$ such that $p(X_1)\subset Y_1$ and $p(X_2)\subset Y_2$.
\par 4) $p^{-1}(0)\cap( X_1\cup X_2)= X_1\cap X_2.$
\par\noindent
Then $P(X_1\cup X_2)=C(X_1\cup X_2)$.
\end{Theorem}
 Stout's theorem is useful tool in studying function algebras (see \cite{D1},\cite{P} and the references given therein.) We shall provide an analogous result for rational function algebras (Theorem \ref{main1}).
 \par Let $\mathcal A$ be a uniform algebra on a compact space $X$.  A point $x\in X$ is a  peak point for $\mathcal A$ if there is a function $f\in \mathcal A$ such that $f(x)=1$ while $\vert f(y)\vert <1$ for
 $y\in X$ and $y\ne x$.  The function $f$ which satisfies this condition is
 called to peak at $x$. The subset $E$ of $X$ is a peak set
for $\mathcal A$ if there is $f \in \mathcal A$ with $f = 1$ on $E$ and $\vert f \vert < 1$ on $X \setminus E$.
\par Clearly, if $x$ is a peak point then $E=\{x\}$ is a peak set. The well known lemma below is
  a simple observation that certain points are peak point for $P(X).$
\begin{Lemma}\label{peak} {\rm (\cite{St})} If $K$ is a compact, polynomially convex subset of the complex plane, then
every boundary point of $K$ is a peak point for the algebra $P(K)$.
\end{Lemma}

\section{Proof of the main result}
\par We begin this section at the following theorem which is a modification of Stout's theorem.
\begin{Theorem}\label{main1} Suppose that:
\par (i) $X_1, X_2$ are compact subset of $\Bbb C^n$ satisfying every continuous functions on $X_i$ ($i=1,2$) are uniformly approximated on $X_i$ by the rational functions with poles off $X_1\cup X_2$.
\par (ii) $Y_2, Y_2$ are polynomially convex sets of $\Bbb C$ such that $E=Y_1\cap Y_2$ are peak set for both $P(Y_1)$ and $P(Y_2)$.
\par (iii) $p$ is a rational function holomorphic on $X_1\cup X_2$ such that $p(X_1)\subset Y_1$ and $p(X_2)\subset Y_2$.
\par (iv) $p^{-1}(E)\cap( X_1\cup X_2)= X_1\cap X_2.$
\par\noindent
Then $R(X_1\cup X_2)=C(X_1\cup X_2)$. In particular, $X_1\cup X_2$ is rationally convex.
\end{Theorem}
\begin{proof} Let $\mu$ be a measure on $X_1\cup X_2$ that is orthogonal to the algebra $R(X_1\cup X_2)$. Then, we have that $\int f d\mu=0$, for every rational functions $f$  with poles off $X_1\cup X_2$.  We have to show that $\mu$ is a zero measure.
\par We first show that  restriction of $\mu$ to $X_1\setminus p^{-1}(E)$ which  is  zero measure. Since $E$ is peak set for $P(Y_2)$, we may find a function $h\in P(Y_2)$ such that $h(z)=1$ for all $z\in E$ and $\vert h(z)\vert<1$ for every $z\in Y_2\setminus E$. By the dominated convergence theorem, we have
$$0=\int (h\circ p)^nfd\mu\to \int_{X_1\setminus \big(\bigcup_{\lambda\in E} p^{-1}(\lambda)\big)} f d\mu $$
for every rational function $f$  with poles off $X_1\cup X_2$. Thus $$\int_{X_1\setminus\big(\bigcup_{\lambda\in E} p^{-1}(\lambda)\big)}  f d\mu=0$$ for every rational functions $f$ poles off $X_1\cup X_2$. In view the condition (i), we may conclude that $\mu\big|_{X_1\setminus p^{-1}(E)}$ is  the zero measure. In the same way,  $\mu\big|_{X_2\setminus \big(\cup_{\lambda\in E} p^{-1}(E)\big)}$ is  the zero measure.  Hence, the measure $\mu$ is  concentrated on the set
$$\aligned &(X_1\cup X_2) \setminus\big([X_1\setminus \big(\cup_{\lambda\in E} p^{-1}(\lambda)\big)]\cup[X_2\setminus \big(\cup_{\lambda\in E} p^{-1}(\lambda)\big)] \big)
\\&= (X_1\cup X_2)\cap (\cup_{\lambda\in E} p^{-1}(\lambda)= X_1\cap X_2.
\endaligned$$
The last equation follows from (iv).  Using again the condition (i), we can deduce that $\mu=0$. It is a equivalent to $R(X_1\cup X_2)=C(X_1\cup X_2)$. In particular, $R\text{-hull}(X_1\cup X_2)=X_1\cup X_2$, or $X_1\cup X_2$ is rationally convex.
\end{proof}
We get the following corollary.
\begin{Corollary}\label{main2} Suppose that:
\par 1) $X_1, X_2$ are compact subset of $\Bbb C^n$ with $C(X_1)=P(X_1)$
and $C(X_2)=P(X_2)$.
\par 2) $Y_1, Y_2$ are polynomially convex sets of $\Bbb C$ such that $E=Y_1\cap Y_2$ are peak set for both $P(Y_1)$ and $P(Y_2)$.
\par 3) $p$ is a rational function holomorphic on $X_1\cup X_2$ such that $p(X_1)\subset Y_1$ and $p(X_2)\subset Y_2$.
\par 4) $p^{-1}(E)\cap( X_1\cup X_2)= X_1\cap X_2.$
\par\noindent
Then $R(X_1\cup X_2)=C(X_1\cup X_2)$. In particular, $X_1\cup X_2$ is rationally convex.
\end{Corollary}
\begin{proof} The result is followed from the fact that $C(X_1)=P(X_1)$
and $C(X_2)=P(X_2)$ imply the condition (i) of Theorem \ref{main1}.
\end{proof}

 We need the following fact which is a modification of lemma in \cite{P}.
\begin{Lemma}\label{bdbs} Let $X$ be a compact subset of $\bc^2$, and let $\pi : \bc^2\to \bc^2$ be difined by $\pi(z,w)=(z^n, w^m)$. Let $\pi^{-1}(X)=X_{11}\cup X_{12}\cup ...\cup X_{nm}$ with $X_{11}$ compact, and $X_{kl}=\{\exp(\fr{2\pi(k-1)}{m}z, \exp(\fr{2\pi(l-1)}{n}w: (z,w)\in X_{11}\}$ for $1\leq k\leq m, 1\leq l \leq n$. If $R(\pi^{-1}(X))=C(\pi^{-1}(X))$ then $R(X)=C(X)$.
\end{Lemma}
\begin{proof} Suppose that $f\in C(X)$. Then $f\circ \pi\in C(\pi^{-1}(X)).$  Since $$R\big( \pi^{-1}(X)\big)= C\big(\pi^{-1}(X)\big)$$ we can seek two polynomials $P, Q$ such that  $$f\circ \pi \approx \frac{P}{Q}, \text {on\; $\pi^{-1}(X)$}$$
and $Q$ has no zero point in $\pi^{-1}(X)$.
In particular , $f\circ \pi\approx \cfrac{P}{Q}$  on $X_{kl}$, for every $1\leq k\leq m$ and $1\leq l\leq n.$ Hence
$$f(z^m, w^n)\approx \cfrac{P(\rho^{k-1}z, \tau^{l-1}w)}{Q(\rho^{k-1}z, \tau^{l-1}w)}:= \frac{P_{kl}(z,w)}{Q_{kl}(z, w)}.$$
Therefore, we  have got the following approximation
\begin{equation}\label{miss3}
\aligned
f(z^m, w^n)& \approx \cfrac{\cfrac{ P_{11}(z,w)+...+ P_{kl}(z,w)+...+ P_{mn}(z,w)}{mn}}{\cfrac{Q_{11}(z,w)+...+ Q_{kl}(z,w)+...+ Q_{mn}(z,w)}{mn}}
\\& \approx \cfrac{ P_{11}(z,w)+...+ P_{kl}(z,w)+...+ P_{mn}(z,w)}{Q_{11}(z,w)+...+ Q_{kl}(z,w)+...+ Q_{mn}(z,w)}
\endaligned
\end{equation}
on $X_{11}$.
Suppose that $P$ and $Q$ have forms
 $P(z,w)=\sum a_{pq} z^pw^q$ and  $Q(z,w)=\sum b_{pq} z^pw^q$
then the right hand of (\ref{miss3}) can be written  $\cfrac{\sum a_{pq} z^{pm}w^{qn}}{\sum b_{pq} z^{pm}w^{qn}}.$  It follows that
 $$f(z^m, w^n)\approx \cfrac{P(z^m, w^n)}{Q(z^m, w^n)},\; \text{on\; $X_{11}$},$$
and so that $R(X)= C(X).$
\end{proof}

\par \noindent {\bf Proof of Theorem 1.1.}
 Let $\pi : \bc^2\to \bc^2$ be difined by $\pi(z,w)=(z^2, w)$. Let $X=\{ \Big(z^2, \cfrac{\overline z}{1+\overline z} \Big):z\in D\}$.  It is easy to check that $z^2,\cfrac{\overline z}{1+\overline z}$ separates points in $D$. Hence, we have $C(X)=C(D)$. Therefore, we reduce to prove $R(X)=R[z^2, \cfrac{\overline z}{1+\overline z}; D]=C(X)$. To do this, by Lemma \ref{bdbs}, we need to shows that
$$R(\pi^{-1}(X))=C(\pi^{-1}(X)).$$
Set $\tilde X=\pi^{-1}(X)$. We have
$\tilde X=X_1\cup X_2,$
where
$$\aligned &X_1=\{\Big(z, \cfrac{\overline z}{1+\overline z} \Big): z\in  D\}
\\&X_2=\{\Big(-z, \cfrac{\overline z}{1+\overline z} \Big): z\in  D\}
\endaligned
$$
Since $X_i, i=1,2$ are totally real sets, we can infer from H\"ormander-Wermer's theorem (\cite{AW}) that $P(X_i)=C(X_i)$ for all $i=1,2$. Consider the rational function
$$p(u, v)=\cfrac{uv}{1+u}.$$
If we choose $D$ small enough then  $p$ has no pole points in $\tilde X$. By a simple computation, we have
$$Y_1:=p(X_1)=\{\cfrac{\vert z\vert^2}{1+\vert z\vert^2+z+\overline{z}}: z\in D\};$$
$$Y_2:= p(X_2)=\{\cfrac{-\vert z\vert^2}{1+\vert z\vert^2+z+\overline{z}}: z\in D\}.$$
 This implies that $Y_1$ is contained in right  real axis and $Y_2$ is contained in left  real axis. Moreover  $Y_1\cap Y_2=\{0\}$.  In particular, $0$ is boundary point of $Y_1 $ and $Y_2$. It follows from Lemma \ref{peak} that $\{0\}$ is peak point of $P(Y_1), P(Y_2)$.  It is easy to check that
$$p^{-1}(0)\cap ((X_1\cup X_2)=(X_1\cup X_2=\{(0, 0)\}.$$
Applying Corollary 3.2, we can deduce that $$R(X_1\cup X_2)=C(X_1\cup X_2)$$
or $$R(\pi^{-1}(X))=C(\pi^{-1}(X)).$$
The theorem is proved.
\begin{Remark} {\rm Since the proof of Theorem 1.1, we have $\pi^{-1}(X)$ is rationally convex. In fact, $\pi^{-1}(X)$ is not polynomially convex (see \cite{OS}).}
\end{Remark}

\end{document}